\documentclass{article}
\usepackage{amsfonts}
\usepackage{amsmath}
\usepackage{amsthm}
\usepackage{amssymb}
\usepackage[colorlinks=true]{hyperref}
\usepackage{url}

\def\@begintheorem#1#2{\par\bgroup{\sc #1\ #2. }\it\ignorespaces}
\def\@opargbegintheorem#1#2#3{\par\bgroup{\sc #1\ #2\ (#3). } \it\ignorespaces}
\def\@endtheorem{\egroup}
\newtheorem{theorem}{Theorem}[section]

\newtheorem{lemma}[theorem]{Lemma}

\title{An improved upper bound on\\the diameters of subset partition graphs}
\author{J. Mackenzie Gallagher and Edward D. Kim\\{\sl University of Wisconsin-La Crosse}}
\date{}

\begin{document}

\maketitle

\begin{abstract}
In 1992, Kalai and Kleitman proved the first subexponential upper bound for the diameters of convex polyhedra. Eisenbrand et al. proved this bound holds for connected layer families, a novel approach to analyzing polytope diameters. Very recently, Todd improved the Kalai-Kleitman bound for polyhedra to $(n-d)^{1+\log_2d}$. In this note, we prove an analogous upper bound on the diameters of subset partition graphs satisfying a property related to the connectivity property of connected layer families.
\end{abstract}

\section{Introduction}

Dongarra and Sullivan cited the simplex algorithm as one of the most important algorithms of the twentieth century (see~\cite{Dongarra:Top-Ten-Algorithms-of-the-Century}). Diameters of polyhedral graphs are relevant in understanding the efficiency of the simplex algorithm for linear programming. Given any linear program and any pivot rule, the number of pivots required by the simplex algorithm using the pivot rule is at least the diameter of graph of the feasibility polyhedron. Motivated by the efficiency of linear programming, Hirsch (see~\cite{Dantzig:Linear}) asked if $n-d$ pivots would suffice to solve any linear program in $d$ variables with $n$ inequalities. This question turned out to be false if the feasibility polyhedron was unbounded (see~\cite{Klee:d-step} by Klee and Walkup), or if the feasibility polyhedron was a bounded polytope but whose simplex path was monotone with respect to the objective function (see~\cite{Todd:The-monotonic-bounded} by Todd). Since these results, the Hirsch Conjecture became the assertion that every $d$-dimensional bounded polytope with $n$ facets has combinatorial diameter at most $n-d$. This statement of the Hirsch Conjecture was disproved by Santos in~\cite{Santos:CounterexampleHirsch} and subsequently improved in~\cite{MatschkeSantosWeibel:Width5Prismatoids} by Matschke, Santos, and Weibel.

Abstractions of polytopes have been studied since the 1970s as an avenue to understanding the diameters of polytopes and polyhedra. (See,~e.g.,~\cite{Adler:AbstractPolytopesThesis}, \cite{Adler:LowerBounds}, \cite{AdlerDantzigMurty:AbstractPolytopes}, \cite{Adler:MaxDiamAbsPoly}, \cite{AdlerSaigal:LongPathsAbstract}, \cite{Kalai:DiameterHeight}, and~\cite{Murty:GraphAbstract}.) More recent studies of abstractions have been studied by Eisenbrand, H\"ahnle, Razborov, and Rothvo\ss{} (see~\cite{Eisenbrand:Diameter-of-Polyhedra}, \cite{Hahnle:Diplomathesis}, and~\cite{Hahnle:SPGs}) and by the second author in~\cite{Bogart:SuperlinearSPG} and~\cite{Kim:PolyhedralGraphAbstractions}. This method of analyzing the diameter of polytopes has gained interest in light of the counterexamples to the Hirsch Conjecture. Though both unbounded polyhedra and bounded polytopes do not satisfy the Hirsch Conjecture in general, no superlinear lower bound is known for their diameters. The non-existence of superlinear diameters is codified in what is known as the Linear Hirsch Conjecture. Recent results by Eisenbrand et al.{}~(see~\cite{Eisenbrand:Diameter-of-Polyhedra}, \cite{Hahnle:SPGs}) prove that statements analogous to the Linear Hirsch Conjecture are false for certain abstractions of polytopes. Recently, Bogart and Kim (see~\cite{Bogart:SuperlinearSPG}) showed that subset partition graphs (a certain abstraction of polyhedra) satisfying all previously studied combinatorial properties for abstract polytopes have superlinear diameter.

For upper bounds, the most general upper bound which applies to convex polytopes and polyhedra is a supexponential bound by Kalai and Kleitman in~\cite{Kalai:Quasi-polynomial} which states that the diameter of any $d$-dimensional polyhedron with $n$ facets is at most $n^{1+\log_2d}$. This upper bound was recently improved to $(n-d)^{\log_2d}$ by Todd in~\cite{Todd:KalaiKleitman}. A result in~\cite{Kim:PolyhedralGraphAbstractions} based on earlier results from Eisenbrand et al.{} in~\cite{Eisenbrand:Diameter-of-Polyhedra} showed that an upper bound of $n^{1+\log_2d}$ applies to the diameters of subset partition graphs satisfying a connectivity property. In this note, we show that an improved upper bound of $(n-d)^{1+\log_2d}$ applies to the diameters of these subset partition graphs. (See Theorem~\ref{theorem:Todd-diameter-upper-bound-SPG}.) For more about the Hirsch Conjecture, see~\cite{Klee:d-step} or the recent survey~\cite{KimSantos:HirschSurvey}. Our terminology on polytopes follows~\cite{Ziegler:Lectures}.

\section{Preliminaries and notation}

Let $S=\{1,2,\dots,n\}$ be a finite set called the \emph{symbol set} (each $s \in S$ is called a \emph{symbol}). We use the notation $\binom{S}{d}$ to denote the set of all $d$-element subsets (or \emph{$d$-sets}) of $S$. Let $G=(\mathcal{V},E)$ be a connected graph with vertex set $\mathcal{V}$ and edge set $E$. For a set $\mathcal{A}\subseteq \binom{S}{d}$ and a vertex set $\mathcal{V}=\{\mathcal{V}_1,\mathcal{V}_2,\dots,\mathcal{V}_k\}$, if $\mathcal{V}$ partitions $\mathcal{A}$ in the sense that	
	\begin{enumerate}
	\item $\mathcal{V}_i \cap \mathcal{V}_j = \emptyset$ for all $i\neq j$
	\item $\mathcal{A}=\mathcal{V}_1 \cup \mathcal{V}_2 \cup \cdots \cup \mathcal{V}_k$,
	\item $\mathcal{V}_i\neq \emptyset$ for all $i$,
	\end{enumerate}
we say $G$ is a \emph{$d$-dimensional subset partition graph on the symbol set $S$}.

Let $G=(\mathcal{V},E)$ be a $d$-dimensional subset partition graph of $\mathcal{A}$ on the symbol set $S$ and let $F\subseteq S$. We define a new subset partition graph $G_F=(\mathcal{V}_F,E_F)$ of $\mathcal{A}_F$ on the symbol set $S$ by applying the following operations on $G$:
	\begin{enumerate}
	\item Remove any $d$-sets in $\mathcal{A}$ that do not contain $F$, resulting in $\mathcal{A}_F = \{A \in \mathcal{A} \mid F \subseteq A\}$.
	\item For each $i$, if $\mathcal{V}_i \cap \mathcal{A}_F = \emptyset$, remove $\mathcal{V}_i$ and any incident edges.
	\end{enumerate}
The resulting graph $G_F$ is called the \emph{restriction} of $G$ with respect to $F$. This process can be thought of as ``inducing" on $F$ to obtain the subgraph $G_F$.
	
Historically, the following properties of subset partition graphs have been considered:
\begin{itemize}
\item {\bf dimension reduction:} if $F \subseteq S$ with $\vert F\vert\leq d$, the restriction graph $G_F$ is connected.
\item {\bf adjacency:} for elements $A,A'$ of $\mathcal{A}$, if $\vert A \cap A'\vert=d-1$, then $A$ and $A'$ are  in the same or adjacent vertices of $G$.
\item {\bf strong adjacency:} adjacency is satisfied and, for any two vertices $\mathcal{V},\mathcal{V'}\in\mathcal{A}$ there exist $d$-sets $A\in \mathcal{V}$ and $A' \in \mathcal{V'}$ with $\vert A\cap A'\vert=d-1$.
\item {\bf endpoint-count:} if $F\in \binom{S}{d-1}$, then the number of vertices in $G$ containing $F$ is no greater than 2.
\end{itemize}
In~\cite{Eisenbrand:Diameter-of-Polyhedra}, Eisenbrand et al. showed that subset partition graphs with dimension reduction whose underlying graphs $G = P_k$ are paths had diameter in $\Omega(\frac{n^2}{\log n})$. In~\cite{Kim:PolyhedralGraphAbstractions}, Kim showed a similar asymptotic lower bound of $\Omega(n^2)$ holds for subset partition graphs satisfying adjacency and endpoint-count, which was subsequently improved by H\"ahnle (see~\cite{Hahnle:SPGs}) to an exponential lower bound for subset partition graphs satisfying all the above properties except dimension reduction. Very recently, Bogart and Kim (see~\cite{Bogart:SuperlinearSPG}) showed that when all four properties are considered, the Eisenbrand asymptotic lower bound of $\Omega(\frac{n^2}{\log n})$ holds.

This drastic difference (exponential versus almost-quadratic) in known diameter lower bounds shows that, given our current understanding, the connectivity property of dimension reduction is the key property in understanding the behavior of polytope diameters. In~\cite{Kalai:Quasi-polynomial}, Kalai and Kleitman proved the diameter of polyhedra is at most $n^{1+\log_2d}$. Eisenbrand et al. (see~\cite{Eisenbrand:Diameter-of-Polyhedra}) showed that this upper bound for polyhedra also holds for subset partition graphs satisfying dimension reduction. In~\cite{Todd:KalaiKleitman}, Todd improved the Kalai-Kleitman bound to $(n-d)^{\log_2d}$. We show that an analogous upper bound of $(n-d)^{1+\log_2d}$ holds for subset partition graphs satisfying dimension reduction in Theorem~\ref{theorem:Todd-diameter-upper-bound-SPG} below.

\section{Maximal diameter of subset partition graphs}

Let $\Sigma_{\text{DR}}(d,n)$ denote the maximum diameter among $d$-dimensional subset partition graphs on $n$ symbols satisfying dimension reduction. In this section, we prove a diameter upper bound result for subset partition graphs which is analogous to Todd's improvement in~\cite{Todd:KalaiKleitman} on the Kalai and Kleitman bound in~\cite{Kalai:Quasi-polynomial}. As in~\cite{Kalai:Quasi-polynomial} and~\cite{Todd:KalaiKleitman}, all logarithms are base $2$. We make use of the fact that if $y,z > 0$, then $y^{\log z} = z^{\log y}$.

\begin{lemma}\label{lemma:inequalities}
Let $c = \log 3$.
If $d \geq 3$ is an integer, then $(\frac{d-1}{d})^{c}  + \frac1d  + \frac{2}{d \cdot d^{\log d}} \leq 1$.
\end{lemma}
\begin{proof}
Let $f(d)=(\frac{d-1}{d})^{c}  + \frac1d  + \frac{2}{d \cdot d^{\log d}}$. Indeed, when $d=3$ and $d=4$, we have $f(3)\approx .976$ and $f(4)\approx .915$. The function $f$ is increasing on the interval $[5,\infty)$ and $f(5)<1$. Since $f(d)\to1$ as $d\to \infty$, it follows that $f(d)\leq 1$ for all $d\geq 5$.
\end{proof}

\begin{lemma}\label{lemma:d2quadratic}
For $n \geq 2$ integer, one has $\Sigma_{\text{DR}}(2,n) \leq (n-2)^2$.
\end{lemma}
\begin{proof}
For $n=3$ and $n=4$, the inequality $\Sigma_{\text{DR}}(2,n) \leq (n-2)^2$ holds by quick inspection. There are at most $\binom{n}{2}-1$ edges to traverse between any two vertices of a $2$-dimensional subset partition graph with $n$ symbols. The results for $n=2$ and $n \geq 5$ follows by comparing the graphs of $y=(x-2)^2$ and $y=\frac{x(x-1)}{2}-1$.
\end{proof}

\begin{theorem}\label{theorem:Todd-diameter-upper-bound-SPG}
For $n \geq d \geq 1$, one has $\Sigma_{\text{DR}}(d,n) \leq (n-d)^{1+\log d}$.
\end{theorem}
\begin{proof}
By Lemma~\ref{lemma:d2quadratic}, the stated inequality holds for $d \leq 2$. In addition, when $n < 2d$, we can ``move to a facet'' in the following sense: because any two $d$-sets share a common symbol, the diameter is at most~$\Sigma_{\text{DR}}(d-1,n-1) \leq (n-d)^{1+\log (d-1)}$ by induction. So, we may assume that $d \geq 3$ and that $n \geq 2d$. This implies that $n - d \geq d$, thus $n-d \geq 3$. The last inequality implies, $\log(n-d) \geq \log_2 3 =: c \approx 1.5$.

The Kalai-Kleitman recursion of $\Delta(d,n) \leq \Delta(d-1,n-1) + 2 \cdot \Delta(d, \lfloor n/2 \rfloor) + 2$ in~\cite{Kalai:Quasi-polynomial} holds for subset partition graphs satisfying dimension reduction due to the Pigeonhole Principle (see~\cite{Eisenbrand:Diameter-of-Polyhedra} or~\cite{Kim:PolyhedralGraphAbstractions}). By induction:
\begin{align*}
\Sigma_{\text{DR}}(d,n) &\leq \Sigma_{\text{DR}}(d-1,n-1) + 2 \cdot \Sigma_{\text{DR}}(d, \lfloor n/2 \rfloor) + 2 \\
&\leq (n-d)^{1+\log(d-1)} + 2 \cdot (n/2 - d)^{1+\log d} + 2\\
&= (n-d) \cdot (n-d)^{\log(d-1)} + 2 \cdot (n/2-d) \cdot (n/2 - d)^{\log d} + 2\\
&= (n-d) \cdot (d-1)^{\log(n-d)} + 2  (n/2-d) \cdot d^{\log(n/2-d)} + 2\\
&= (n-d) \left(\frac{d-1}{d}\right)^{\log(n-d)} d^{\log(n-d)} + 2 (n/2-d) \cdot d^{\log(n/2-d)} + 2\\
&\leq (n-d)\left(\frac{d-1}{d}\right)^{\log(n-d)} d^{\log(n-d)} + (n-d) \cdot d^{\log((n-d)/2)} + 2 \\
&= (n-d)\left(\frac{d-1}{d}\right)^{\log(n-d)} d^{\log(n-d)} + \frac1d (n-d) \cdot d^{\log(n-d)} + 2\\
&\leq (n-d) \left(\frac{d-1}{d}\right)^{c} d^{\log(n-d)} + \frac1d (n-d) \cdot d^{\log(n-d)} + 2\\
&= (n-d) d^{\log(n-d)} \left[ \left(\frac{d-1}{d}\right)^{c}  + \frac1d   + \frac{2}{(n-d) d^{\log(n-d)}} \right]\\
&\leq (n-d) d^{\log(n-d)} \left[ \left(\frac{d-1}{d}\right)^{c}  + \frac1d  + \frac{2}{d \cdot d^{\log d}} \right].
\end{align*}
Since $(\frac{d-1}{d})^{c}  + \frac1d  + \frac{2}{d \cdot d^{\log d}} \leq 1$ for all $d \geq 3$
by Lemma~\ref{lemma:inequalities}, the previous expression is bounded above by
$(n-d)d^{\log(n-d)} = (n-d)(n-d)^{\log(d)} = (n-d)^{1+\log d}$.
\end{proof}
As a corollary, the diameter bound of $(n-d)^{1+\log d}$ holds for the connected layer families of Eisenbrand et al. in~\cite{Eisenbrand:Diameter-of-Polyhedra}. For the diameters of subset partition graphs satisfying dimension reduction, the gap between the upper bound of $(n-d)^{1+\log d}$ and the Eisenbrand asymptotic lower bound of $\Omega(\frac{n^2}{\log n})$ is rather large. It remains to be seen if H\"ahnle's conjectured upper bound of $d(n-1)$ for polytopes (see, e.g.,~\cite{Santos:RecentProgressCombinatorialDiameter}) holds for subset partition graphs with dimension reduction.

\end{document}